 \documentclass[a4paper,12pt]{amsart}
\usepackage[margin=2cm]{geometry}
 \usepackage{amsmath,amsthm,amssymb}
 \usepackage{graphicx,color}
 \usepackage[all]{xy}

\theoremstyle{definition}
\newtheorem{thm}{Theorem}[section]
\newtheorem{lem}[thm]{Lemma}
\newtheorem{cor}[thm]{Corollary}
\newtheorem{prop}[thm]{Proposition}
\newtheorem{ques}[thm]{Problem}

\newtheorem{rem}[thm]{Remark}
\newtheorem{ex}[thm]{Example}
\numberwithin{equation}{section}

\newcommand{\Z}{\mathbb{Z}}
\newcommand{\R}{\mathbb{R}}
\newcommand{\C}{\mathbb{C}}

\newcommand{\Hom}{\mathrm{Hom}}

\def\red#1{{\textcolor{black}{#1}}}

\begin{document}
\title{Representing a point and the diagonal as zero loci in flag manifolds}
\author{Shizuo KAJI}
\thanks{This work was partially supported by KAKENHI, Grant-in-Aid for Scientific Research (C) 18K03304.}
\date{}
\address{Institute of Mathematics for Industry, Kyushu University, Japan}
\email{skaji@imi.kyushu-u.ac.jp}


\begin{abstract}
The zero locus of a generic section of a vector bundle over a manifold
defines a submanifold.
A classical problem in geometry asks 
to realise a specified submanifold in this way.
We study two cases; a point in a generalised flag manifold and 
the diagonal in the direct product of two copies of a generalised flag manifold.
These cases are particularly interesting since they are related to
ordinary and equivariant Schubert polynomials respectively.
\end{abstract}

%

\maketitle

\section{Introduction}
Let $N$ be a manifold of dimension $2n$.
Consider a smooth function $f: N\to \C^m$ having $0\in \C^m$ as a regular value.
Then, $M=f^{-1}(0)\subset N$ is a submanifold of codimension $2m$.
Conversely, we can ask if a submanifold $M\subset N$ of codimension $2m$
can be realised in this way, or more generally,
as the zero locus of a generic section of a rank $m$ complex vector bundle $\xi\to N$.
Here by a generic section, we mean it is transversal to the zero section.
We say $M$ is \emph{represented by $\xi$} if such a bundle $\xi$ exists.

The following example tells that even for the simplest case the question is not as trivial as it may appear to be.
\begin{ex}
Consider the representability of a point in $S^2$. 
Identify $S^2=\C P^1$ and let $\gamma^*\to \C P^1$ be the dual of the tautological bundle
\red{
\[
\gamma =\{ ([z_0,z_1],(cz_0,cz_1)) \mid [z_0,z_1]\in \C P^1, z_0z_1\neq0, c\in \C\}.
\]
One of its generic sections is given by the projection $(cz_0,cz_1) \mapsto cz_0$,
whose zero locus is exactly the south-pole $[0,1]\in \C P^1$.}
Since $S^2$ is homogeneous with a transitive $SO(3)$ action,
for any pair of points $x,y\in S^2$, there is an element $g\in SO(3)$ such that
$gx=y$. By choosing $g$ appropriately, we can represent any point in $S^2$ by $g^*(\gamma^*)$.

On the other hand, consider the representability of a point in $S^{2n}$ for $n>2$.
Bott's integrality theorem tells that the top Chern class $c_n$ of any rank $n$ complex vector bundle
 on $S^{2n}$ is divisible by $(n-1)!$ (see, for example, \cite[Proposition 6.1]{almost-complex}).
However, if there is a rank $n$ bundle with a generic section whose zero locus is a point,
the top Chen class of the bundle has to be the generator of $H^{2n}(S^{2n};\Z)$.
Hence, there is no such bundle, and a point in $S^{2n}$ is representable if and only if $n=1$. 
\end{ex}

From now on, all spaces are assumed to be based, and the base points are denoted by $pt$.
The following two submanifolds are particularly interesting (see \cite{PSP} and references therein): 
\begin{enumerate}
\item the base point $\{pt\}\subset X$
\item the diagonal $\red{\Delta(X)=\{(x,x)\mid x\in X\} }\subset X\times X$.
\end{enumerate}
\red{In the language of~\cite{PSP}, if any point in $X$ (resp. the diagonal in $X\times X$) is representable,
 $X$ is said to have property $(P_c)$ (resp. $(D_c)$).}
Note that the choice of the base point does not make any difference \red{when $X$ is connected} since
for any pair of points $x,y\in X$ there exists a diffeomorphism $f:X\to X$ satisfying $f(y)=x$ so that 
the bundle and the section for the representability of the point $x$ are pulled back to represent the point $y$.
Note also that when $\Delta(X)\subset X\times X$ is represented by $\xi$, 
then $\{pt\} \subset X$ is represented by $\iota^*(\xi)$, where
$\iota$ is the inclusion $N \hookrightarrow N\times N$ defined by $\iota(x)=(x,pt)$.

In~\cite{PSP} analogous problems in different settings are considered; 
in an algebraic setting and in topological settings with complex bundles, real bundles, and real oriented bundles.
In this note, we focus on the following topological variant:
\begin{ques}
Let $X$ be a (generalised) flag manifold $G/P$, where $G$ is a complex, \red{connected, simple Lie group and $P$ is one parabolic subgroup.}
Find a rank $\dim_\C(X)$ complex bundle $\xi\to X$ (resp. $\xi\to X\times X$)
with a smooth generic section which vanishes exactly at the base point (resp. along $\Delta(X)$).
\end{ques}

The problem is related to Schubert calculus.
The Poincar\'e dual to the fundamental class of the base point defines a cohomology class, which corresponds to the top Schubert class.
Similarly, the class of the diagonal can be thought of as a certain restriction of the torus equivariant top Schubert class (see \S~\ref{schubert}).
\medskip

\red{
Fulton showed a remarkable result (\cite[Proposition 7.5]{F}) that
the diagonal in $X\times X$ for any type $A$ flag manifold $X=SL(k+1)/P$ with any parabolic subgroup $P$
is representable. 
Note that Fulton's result works in a holomorphic setting and is stronger than our topological setting.
On the other hand, for full (i.e. complete) flag manifolds (when $P=B$ is the Borel subgroup) of other Lie types, 
P. Pragacz and the author showed (\cite[Theorem 17]{kapr})
that 
the base point in $G/B$ (and hence, the diagonal in $G/B\times G/B$) is not representable unless $G$ is of type $A$ or $C$.
Indeed, we show in this note
\begin{thm}[Proposition \ref{prop:point}, Theorem \ref{exceptional}]
The base point (resp. the diagonal) of $G/B$ is representable if and only if $G$ is of type $A$ or $C$.
Moreover, 
the base point of $G/P$ is not representable for any proper parabolic subgroup $P$
when $G$ is of exceptional type.
\end{thm}}
Thus, the remaining cases are those of flag manifolds of type $B$, $C$, and $D$.
In \cite[Theorem 12]{PSP}, non representability of the diagonal is shown for the odd complex quadrics, 
which are partial flag manifolds of type $B$.
Naturally, we may ask if there is any flag manifold where the base point is representable but the diagonal is not.
The main result of this note is to give such an example. 
\red{
Namely, we show 
\begin{thm}[Theorem \ref{main-thm}]
Let $Lag_\omega(\C^{2k})$ be the Lagrangian Grassmannian of maximal isotropic subspaces 
in the complex symplectic vector space $\C^{2k}$ with a symplectic form $\omega$.
The base point in $Lag_\omega(\C^{2k})$ is representable for any $k$,
but its diagonal is not when $k\equiv 2 \mod 4$.
\end{thm}
}
We also see that the base point is not representable for many type $B$ and $D$ partial flag manifolds
(Proposition \ref{orthogonal-Grassmann} and Remark \ref{rem:final}).

Throughout this note, $H^*(X)$ stands for the singular cohomology $X$ with integer coefficients.
Denote by $M\subset N$ a closed oriented submanifold $M$ of codimension $2m$ embedded in a closed oriented $2n$-manifold $N$.
The cohomology class which is Poincar\'e dual to the fundamental class of $M$ is denoted by $[M]\in H^{2m}(N)$.

\subsection*{Acknowledgement}
The author would like to thank 
the referee for his/her careful reading and valuable comments.

\section{Criteria for representability}
We begin with trivial but useful criteria for the representability of submanifolds in general.
\begin{prop}\label{necessary}
Let $N$ be a closed oriented manifold of dimension $2n$.
Assume that $\xi\to N$ represents a submanifold $M\subset N$ of codimension $2m$.
\begin{enumerate}
\item The top Chern class $c_m(\xi)\in H^{2m}(N)$ is equal to the class $[M]$.
\item The restriction $\xi|_M$ is isomorphic to the normal bundle $\nu(M)$ of $M\subset N$.
\end{enumerate}
\end{prop}
Note that the converse to (1) does not hold; 
the equality of classes \red{$c_m(\xi)=[M]$} does not necessarily mean that we can find a generic section whose zero locus is exactly $M$.
However, when $M$ is a point, we can pair zeros with opposite orientations of any generic section to cancel out.
This means:
\begin{lem}\label{lem:point}
The base point in a closed oriented $2n$-manifold $N$ is representable if and only if there is a rank $n$ complex bundle
whose top Chern class is the generator of $H^{2n}(N)$.
\end{lem}
The complex $K$-theory $K^0(N)$ can be identified with the set of stable equivalence classes of vector bundles over $N$, where
\[
 \xi_1 \sim \xi_2 \Leftrightarrow \xi_1 \oplus \C^l \simeq \xi_2 \oplus \C^l \text{ for some } l\ge 0.
\]
The Chern class $c_m(\xi)\in H^{2m}(N)$ of a bundle $\xi\to N$ depends only on the stable equivalence class of $\xi$ in $K^0(N)$.
Therefore, if for a complete set of representatives of $K^0(N)$ there is no bundle whose $m$-th Chern class is equal to
the class $[M]\in H^{2m}(N)$, we can conclude that $M$ is not representable.
 
\section{Flag manifolds}\label{schubert}
From now on, we focus on flag manifolds $G/P$ with the base point taken to be the identity coset $eP$.
We assume $G$ is a complex, connected, simple Lie group with a fixed Borel subgroup $B$ containing a maximal compact torus $T\subset B$, and its Weyl group is denoted by $W(G)$.
A parabolic subgroup $P$ is a \red{closed} subgroup of $G$ containing $B$.
Parabolic subgroups are in one-to-one correspondence with the subgraphs of the Dynkin diagram of $G$.
Denote by $K$ the maximal compact subgroup of $G$ containing $T$ and by $H$ its subgroup $P\cap K$.
We have a diffeomorphism $K/H \simeq G/P$ by the Iwasawa decomposition, and in particular, $K/T\simeq G/B$.
We use notations $G/P$ and $K/H$ interchangeably.
We also have a diffeomorphism $K/H\simeq \widetilde{K}/p^{-1}(H)$, where $p:\widetilde{K}\to K$ is the universal covering.
So we can assume $K$ is simply-connected if necessary.

The universal flag bundle is denoted by
$K/T \overset{c}{\hookrightarrow} BT \to BK$, where $BK$ is the classifying space of $K$.
More generally, we have the universal partial flag bundle $K/H\hookrightarrow BH\to BK$.
We say a bundle $K/H\to E \to X$ is a flag bundle if it is a pullback of the universal (partial) flag bundle
via a map $X\to BK$.
The Atiyah-Hirzebruch homomorphism $H^*(BT)\to K^0(K/T)$ is defined by 
 assigning to a character $\lambda\in \Hom(T,\C^*)\simeq H^2(BT)$ the line bundle $L_\lambda := K\times_T \C_\lambda$ over $K/T$
 and extending multiplicatively. Here denoted by $\C_\lambda$ is the space $\C$ acted by $T$ via $\lambda$.
 This map is known to be surjective when $K$ is simply-connected  (cf.~\cite{KK}).

\smallskip
We first note how the representability of the base point and the diagonal is related to Schubert polynomials.
One way to look at Schubert polynomials \cite{BGG} is that they are elements in $H^*(BT)$ which 
pullback via $c:K/T \to BT$ to the classes of Schubert varieties in $K/T$.
In other words, they are polynomials in the first Chern classes of line bundles on $K/T$ representing the Schubert classes.
The top Schubert polynomial represents the class of the base point and it is known by \cite{BGG}
that it ``produces'' all the other Schubert polynomials when applied the divided difference operators.
So in a sense, the top Schubert polynomial carries the information of the whole $H^*(K/T)$.
This is why we are interested in representing the base point.
A similar story goes for the (Borel) $T$-equivariant cohomology $H^*_T(K/T)$, the top double Schubert polynomial,
and the diagonal, as is explained below.

Let $EK$ be the universal $K$-space; that is, $EK$ is contractible on which $K$ acts freely.
\red{Then, the classifying spaces are taken to be $BK=EK/K$ and $BT=EK/T$. 
The Borel construction $K/T$ is defined to be $EK\times_T K/T$, where
$[x,gT]=[xt^{-1},tgT]\in EK\times_T K/T$ for $t\in T$.
Consider the following commutative diagram:
\[
\xymatrix{
 & K/T \ar@{=}[r] \ar[d] & K/T \ar[d]^c \\
K/T \ar[r] \ar@{=}[d] & EK\times_T K/T \ar[r]^(.65){p_2} \ar[d]^{p_1} & BT \ar[d] \\
K/T \ar[r]^c & BT \ar[r] & BK,
}
\]
where the lower-right squares is a pullback, $p_1([x,gT])=[x]$, and $p_2([x,gT])=[xg]$.
We have the sequence of maps
$
K/T \times K/T \xrightarrow{i} EK\times_T K/T \xrightarrow{p} BT\times BT,
$
where $i(g_1T,g_2T)=[pt\cdot g_2,g_2^{-1}g_1T]$ and $p=(p_1,p_2)$. }
The class of the equivariant point $EK\times_T eT/T$ pulls back via $i$ to the class of the diagonal in $K/T\times K/T$.
The class of $EK\times_T eT/T$ corresponds to the class of the top Schubert variety in the equivariant cohomology,
which in turn corresponds to the top double Schubert polynomial.

Let us look at the concrete example of $K=U(k)$. Note that $U(k)/T\simeq SU(k)/T'$ with $T'=T\cap SU(k)$.
Although $U(k)$ is not simple, we consider $U(k)$ for convenience.
Lascoux and Sch\"utzenberger's top double Schubert polynomial~\cite{LS} 
\[
\mathfrak{S}_{w_0}(x,y)=\prod_{1\le i<j\le k}(x_i-y_j)
\]
can be considered as an element in 
$H^*(BT\times BT)\simeq H^*(BT)\otimes H^*(BT) \red{\simeq \Z[x_1,\ldots,x_k]\otimes \Z[y_1,\ldots,y_k]}$,
which pulls back via $p$ to the class of the equivariant point $EK\times_T eT/T$ in the equivariant cohomology
$H^*(EK\times_T K/T)= H^*_T(K/T)$.
This class further pulls back via $i$ to the class of the diagonal in $H^*(K/T\times K/T)$.
For a character $\lambda\in H^2(BT)$, let $\hat{L}_\lambda$ be the line bundle $ET\times_T \C_\lambda\to BT$
\red{such that $c_1(\hat{L}_\lambda)=\lambda$.}
As $\mathfrak{S}_{w_0}(x,y)$ is a product of linear terms, we can define the 
rank $\dim_\C(K/T)=k(k-1)/2$ bundle
\[
\xi = \bigoplus_{1\le i<j\le k} \hat{L}_{x_i} \hat{\otimes} \hat{L}^*_{y_{j}} \to BT\times BT
\]
such that its top Chern class is equal to $\mathfrak{S}_{w_0}(x,y)$.
We have $c_{k(k-1)/2}(i^*p^*(\xi))=[\Delta(K/T)] \in H^{k(k-1)}(K/T\times K/T)$.

Similarly for $K=Sp(k)$, consider the rank $\dim_\C(K/T)=k^2$ bundle 
\[
 \xi=\bigoplus_{1\le i \le j\le k} \hat{L}_{x_i}\hat{\otimes} \hat{L}_{y_j}
 \bigoplus_{1\le i < j\le k}\hat{L}_{x_i}\hat{\otimes}\hat{L}^*_{y_j} \to BT\times BT.
\]
Then, the top Chern class of $p^*(\xi)$ is the equivariant top Schubert class~\cite[\S 8]{K-hiller}.
We have $c_{k^2}(i^*p^*(\xi))=[\Delta(K/T)] \in H^{2k^2}(K/T\times K/T)$.
This means that there is a generic section $s$ of $\xi$ such that $[Z(s)]=[\Delta(K/T)]$ but 
\red{this does not imply the existence of $s$ such that $Z(s)$ is exactly $\Delta(K/T)$.}
We will show that the diagonal of $Sp(k)/T$ is actually representable.

For this, we recall the following slight generalisation of Fulton's result (\cite[Proposition 7.5]{F}) in the current smooth setting.
\begin{prop}[{\cite[Theorem 14]{kapr}}]\label{bundle}
If a point in (resp. the diagonal of) $X$ is representable, 
then so is any point in (resp. the diagonal of) the total space $E$ of any flag bundle of type $A$:
\[
SL(k)/P \hookrightarrow E \xrightarrow{p} X,
\]
where $P$ is any parabolic subgroup of $SL(k)$.
\end{prop}
\begin{proof}
\red{
We reproduce a proof for the diagonal case here for the sake of completeness.
By universality, $E\xrightarrow{p} X$ is the pullback of the universal flag bundle along some map $f: X\to BSL(k)$.
The idea is to consider the following pullback diagram
\[
\xymatrix{
E\times E \ar[rr] \ar[d]^{p\times p} & & BP\times BP \ar[d] \\
X\times X \ar[rr]^(0.4){f\times f} & & BSL(k)\times BSL(k)
}
\]
and construct a bundle over $E\times E$ as the sum of pullbacks of bundles over $BP\times BP$ and $X\times X$.
Let $G$ be $SL(k)$ and
$G/P \to BP \to BG$ be the universal flag bundle, whose fibre we identify with the space of flags
$\{ 0\subset  U_1 \subset U_2 \subset \ldots \subset U_l \subset \C^{k} \}$.
We have the corresponding tautological sequence of bundles on $BP$:
\[
 \mathcal{U}_1 \stackrel{\iota}{\hookrightarrow} \mathcal{U}_2 \stackrel{\iota}{\hookrightarrow} \cdots \stackrel{\iota}{\hookrightarrow}
 \mathcal{U}_l \stackrel{\iota}{\hookrightarrow} \gamma_{k}
 \xrightarrow{q} \mathcal{V}_1 \xrightarrow{q} \cdots \xrightarrow{q} \mathcal{V}_l,
\]
where $\gamma_{k}$ is the pullback of the universal vector bundle over $BG$ and $\mathcal{V}_i=\gamma_{k}/\mathcal{U}_i$.
Denote by $\pi_1,\pi_2: BP\times BP \to BP$ the left and the right projections.
The following rank $\dim_\C(G/P)$ bundle over $BP\times BP$ is defined in \cite[Proposition 7.5]{F}:
\[
\xi_{BP}=\left\{ \bigoplus_{i=1}^l h_i\in \bigoplus_{i=1}^l \Hom(\pi_1^*(\mathcal{U}_i),\pi_2^*(\mathcal{V}_i)) 
	\mid \forall i, \ q\circ h_i = h_{i+1}\circ \iota  \right\}.
\]
Restricted on the fibre product $BP\times_{BG} BP\subset BP\times BP$,
it admits a section $s_{BP}: BP\times_{BG} BP\to \xi_{BP}$
defined by the tautological map $\pi_1^*(\mathcal{U}_i) \to \pi_1^*(\gamma_{k})=\pi_2^*(\gamma_{k})\to \pi_2^*(\mathcal{V}_i)$,
which vanishes exactly along the diagonal $\Delta(BP)\subset BP\times_{BG} BP$.
By a partition of unity argument, we can extend $s_{BP}$ to the whole $BP\times BP$, which
we denote by the same symbol $s_{BP}$
(note that this is the place where we have to work in our smooth setting unlike Fulton's original work in the holomorphic setting).
}

\red{
Let $\xi_X$ be a bundle over $X\times X$ with a generic section $s_X$ which represents the diagonal of $X$.
Note that the pullback bundle $(p\times p)^*(\xi_X)$ admits a section $(p\times p)^*(s_X)$ 
whose zero locus is $(p\times p)^{-1}(\Delta(X))=E\times_X E$.
The bundle $\xi=(p\times p)^*(\xi_X)\oplus (f\times f)^*(\xi_{BP})$ over $E$ has rank $\dim(E)/2=(\dim(G/P)+\dim(X))/2$.
The section of $\xi$ defined by $(p\times p)^*(s_X)\oplus (f\times f)^*(s_{BP})$ 
 vanishes exactly along the diagonal as the following is a pullback
 \[
 \xymatrix{
 E\times_X E \ar[r]^(0.45){f\times f}  & BP \times_{BG} BP\\
 \Delta(E) \ar[r] \ar[u] & \Delta(BP). \ar[u]
 }
 \]
 }
\end{proof}

\red{
\begin{prop}\label{prop:point}
When $G$ is of type $C$, the diagonal of $G/P$ for any parabolic subgroup $P$ 
of type $C$ (including $P=B$) is representable.
Consequently, the base point in $G/B$ (resp. the diagonal in $G/B\times G/B$) is representable if and only if $G$ is of type $A$ or $C$.
\end{prop}
}
\begin{proof}
\red{For $1\le k'\le k$, let $F^k_{1,2,\ldots,k'}=Sp(k)/(T^{k'}\cdot Sp(k-k'))$ be
the isotropic flag manifold with respect to a symplectic form $\omega$ in $\C^{2k}$:
\[
F^k_{1,2,\ldots,k'}=\{ 0\subset U_1 \subset U_2 \subset \cdots \subset U_{k'} \subset
 U_{k'}^\perp \subset \cdots U_1^\perp \subset \C^{2k}\mid \dim_\C(U_i)=i\},
\]
where $U_i^\perp=\{v\in \C^{2k}\mid \omega(u,v)=0 \ \forall u\in U_i\}$.}
Denote the tautological bundle on $F^k_{1,2,\ldots,k'}$ corresponding to $U_i$ by $\mathcal{U}_i$.
By dropping $U_{k'}$, we obtain a projection $p: F^k_{1,2,\ldots,k'}\to F^k_{1,2,\ldots,k'-1}$, 
which makes $F^k_{1,2,\ldots,k'}$ the projectivisation of 
$\mathcal{U}_{k'-1}^\perp/\mathcal{U}_{k'-1}$ over $F^k_{1,2,\ldots,k'-1}$.
By Proposition \ref{bundle}, if the diagonal of $F^k_{1,2,\ldots,k'-1}$ is representable, so is that of $F^k_{1,2,\ldots,k'}$.
This procedure can be iterated to $F^k_{1}=\C P^{2k-1}$, of which the diagonal is representable since it is a type $A$ partial flag manifold.

\red{For a full flag manifold of an arbitrary type,
as is reviewed in Introduction, the base point in $G/B$ is not representable 
unless $G$ is of type $A$ or $C$ (\cite[Theorem 17]{kapr}), and
the diagonal is representable when $G$ is of type $A$ (\cite[Proposition 7.5]{F}).
Thus, the second statement holds from the first.
}
\end{proof}

For exceptional Lie groups, the arguments in \cite[\S 6]{kapr} extend to show:
\begin{thm}\label{exceptional}
When $G$ is of exceptional type, the base point in $G/P$ is not representable for any (proper) parabolic subgroup $P$
(including $P=B$).
\end{thm}
\begin{proof}
By taking the universal covering, we can assume $K$ is simply-connected.
Let $H=P\cap K$.
We shall see that there is no bundle $\xi$ with $c_{n}(\xi)=u_{2n} \in H^{2n}(K/H)$,
where $u_{2n}$ is the generator of the top degree cohomology.
The flag bundle $H/T \hookrightarrow K/T \to K/H$ induces isomorphisms
\[
H^*(K/H)\simeq H^*(K/T)^{W(H)}, \quad K^0(K/H)\simeq K^0(K/T)^{W(H)},
\]
where $W(H)$ is the Weyl group of $H$.
The universal flag bundle $K/T \overset{c}{\hookrightarrow} BT \to BK$ 
induces a map $c^*: H^*(BT)\to H^*(K/T)$, which is compatible with the action of $W(K)$.
The Atiyah-Hirzebruch homomorphism
$H^*(BT)\to K^0(K/T)$ is also compatible with the action of $W(K)$ and it restricts to a surjection
$H^*(BT)^{W(H)} \to K^0(K/T)^{W(H)}\simeq K^0(K/H)$.
This asserts that any bundle over $K/H$ stably splits into line bundles when pulled back via $K/T\to K/H$,
and hence, its Chern classes are polynomials in the elements of $H^2(K/T)\simeq c^*(H^2(BT))$.
Let $\tau_{K/H}$ be the smallest positive integer such that 
$\tau_{K/H}\cdot u_{2n}$ is in the image of
\[
c^*: H^*(BT)^{W(H)} \to  H^*(K/T)^{W(H)}\simeq H^*(K/H)
\]
induced by $c^*: H^*(BT)\to H^*(K/T)$. 
Consider the flag bundle
\[
H/T \hookrightarrow K/T \to K/H.
\]
There is a class $v\in H^*(K/T)$ which restricts to the class of the base point in $H^*(H/T)$.
Since the class of the base point in $H^{2n}(K/T)$ is the product of the pullback of $u_{2n}$
with $v$, we have
\[
\tau_{K/T} \le \tau_{H/T} \cdot \tau_{K/H}.
\]
On the other hand, it is known (see \cite{T})
\red{
\begin{align*}
\tau_{SU(k)/T}=\tau_{Sp(k)/T}=1,
\tau_{Spin(k)/T}=\begin{cases} 2 & (7\le k\le 12) \\ 4 & (k=13,14) \\ 8 & (k=15,16), \end{cases} \\
\tau_{G_2/T}=2, \tau_{F_4/T}=6, \tau_{E_6/T}=6, \tau_{E_7/T}=12, \tau_{E_8/T}=2880.
\end{align*}
}
Parabolic subgroups are in one-to-one correspondence with subgraphs of the Dynkin diagram.
So for any (proper) parabolic subgroup of an exceptional Lie group, we can see $\tau_{K/T}>\tau_{H/T}$ from the list above.
Therefore, $t_{K/H}>1$ and $u_{2n}$ cannot be the Chern class of a bundle.
\end{proof}

\section{Grassmannian manifolds}
An argument similar to the one in the previous section also works for some $G/P$ with $G$ of classical types.
\red{ 
Due to the low rank equivalences 
$A_1=B_1=C_1=D_1, B_2=C_2, D_2=A_1\times A_1$,
and $D_3=A_3$, we assume $k>2$ for $B_k$ and $k>3$ for $D_k$. }

\begin{prop}\label{orthogonal-Grassmann}
\red{When $G$ is of type $B_k \ (k>2)$ or $D_k \ (k>3)$
and $P$ is a parabolic subgroup of type $A$, 
then the base point in $G/P$ is not representable.
In particular,
the base point in the maximal orthogonal Grassmannian $OG_k(\C^{2k})$
of maximal isotropic subspaces in the complex quadratic vector space $\C^{2k}$ is representable if and only if $k\le 3$.}
\end{prop}
\begin{proof}
\red{
If the base point in $G/P$ is representable, so would be in $G/B$ by
Proposition \ref{bundle} applied to the flag bundle $P/B\hookrightarrow G/B \to G/P.$
The first statement follows from Proposition \ref{prop:point}.
}

\red{
For the second statement, recall from \cite[\S 1.7]{M} that 
the connected component of $OG_k(\C^{2k})$ containing the identity is
diffeomorphic to the flag manifold $SO(2k)/U(k) \simeq SO(2k-1)/U(k-1)$.
Thus, the base point is representable if and only if $k\le 3$.
}
\end{proof}

The base point is representable in $G/P$ if the diagonal is representable in $G/P\times G/P$.
The following example shows that the converse is not always true.
\begin{thm}\label{main-thm}
Let $Lag_\omega(\C^{2k})\simeq Sp(k)/U(k)$ be the complex Lagrangian Grassmannian of maximal isotropic subspaces 
in the complex symplectic vector space $\C^{2k}$ with a symplectic form $\omega$ (see \cite[\S 1.7]{M}).
\begin{enumerate}
\item The base point in $Lag_\omega(\C^{2k})$ is representable.
\item When $k\equiv 2 \mod 4$, the diagonal in $Lag_\omega(\C^{2k})\times Lag_\omega(\C^{2k})$ is not representable.
\end{enumerate}
\end{thm}
Note that there is a $p$-local homotopy equivalence
$Sp(k)/U(k)\simeq_{p} SO(2k+1)/U(k)$ for odd primes $p$ \red{(\cite{Harris})}, 
so $2$-torsion plays an important role in our problem.
Our proof of the theorem is based on the Steenrod operations, which is similar to that of~\cite[Theorem 11]{PSP}.
We need a few lemmas.

\begin{lem}\label{lem:tangent}
The tangent bundle of a flag manifold $K/H$ is 
\[
    T(K/H)=
\bigoplus_{\beta\in \Pi^+ \setminus \Pi^+_H} L_\beta,
\] 
where $\Pi^+$ (resp. $\Pi^+_H$) is the set of positive roots of $K$ (resp. $H$).
In particular for $Sp(k)/U(k)$, we can take
$\Pi^+=\{2x_i\mid 1\le i\le k\}\cup \{x_i\pm x_j\mid 1\le i<j\le k\}$ and
$\Pi^+_H=\{x_i- x_j\mid 1\le i<j\le k\}$, hence we have
\[
    T(Sp(k)/U(k))\simeq \left(\bigoplus_i L_{2x_i}\right) \oplus \left( \bigoplus_{i<j} L_{x_i+x_j} \right).
\]
\end{lem}
\begin{proof}
    The assertion follows from the standard isomorphism $T(K/H)\simeq K\times_H (L(K)/L(H))$, where
    $L(K)$ and $L(H)$ are Lie algebras of $K$ and $H$ respectively.
\end{proof}

Let $2n=\dim(Sp(k)/U(k))=k(k+1)$.
\begin{lem}[\red{c.f. \cite{P}}]
\label{lem:cohomology}
Let $c_i$ (resp. $q_i$) be elementary symmetric functions
in $x_j$ (resp. $x_j^2$), where $H^*(BT)=\Z[x_1,\ldots,x_k]$.
Then, 
\[
H^*(Sp(k)/U(k)) \simeq \dfrac{\Z[c_1,c_2,\ldots,c_k]}{(\Z[q_1,q_2,\ldots,q_k])^+},
\]
where $(\Z[q_1,q_2,\ldots,q_k])^+$ is the ideal of positive degree polynomials in $q_j$.
In particular,
\begin{align*}
u_{2n} &= \prod_i x_i \prod_{i<j} (x_i+x_j) =\prod_{i=1}^k c_i \\
u_{2n-2} &= \prod_{i=2}^k c_i \\
u_1 &= c_1
\end{align*}
are generators of $H^{2n}(Sp(k)/U(k)),H^{2n-2}(Sp(k)/U(k))$, and $H^{2}(Sp(k)/U(k))$ respectively.
\end{lem}
\begin{proof}
Let $X=Sp(k)/U(k)$.
    Since $H_*(Sp(k))$ has no torsion, by \cite{Borel} we have
\[
H^*(X)\simeq \dfrac{H^*(BT)^{W(U(k))}}{(H^+(BT)^{W(Sp(k))})},
\]
where $(H^+(BT)^{W(Sp(k))})$ is the ideal generated by the positive degree Weyl group invariants.
Since $W(U(k))\curvearrowright H^*(BT)$ is permutation and
$W(Sp(k))\curvearrowright H^*(BT)$ is signed permutation, we have
\[
H^*(X) \simeq \dfrac{\Z[c_1,c_2,\ldots,c_k]}{(\Z[q_1,q_2,\ldots,q_k])^+}.
\]
By the degree reason, it is easy to see $\prod_{i=1}^k c_i \in H^{2n}(X)$,
$\prod_{i=2}^k c_i \in H^{2n-2}(X)$, and 
$c_1\in H^{2}(X)$ are generators.
The Euler characteristic $\chi(X)$ is equal to $\dfrac{|W(Sp(k))|}{|W(U(k))|}$
as the cells in the Bruhat decomposition of $X$ are indexed by the cosets $W(Sp(k))/W(U(k))$.
Since 
\[
    \prod_i (2x_i) \prod_{i<j} (x_i+x_j)=c_n(TX)=\chi(X)u_{2n}=\dfrac{|W(Sp(k))|}{|W(U(k))|}u_{2n}=2^k u_{2n},
    \]
we have $u_{2n} = \prod_i x_i \prod_{i<j} (x_i+x_j)$.
\end{proof}

\begin{lem}\label{lem:spin}
When $k\equiv 2 \mod 4$,
any bundle $\xi\to Sp(k)/U(k)$ representing the base point in $Sp(k)/U(k)$ is spin.
\end{lem}
\begin{proof}
We show $c_n(\xi)=\pm u_{2n}$ implies
$c_1(\xi)\equiv w_2(\xi)=0 \mod 2$, where $w_2(\xi)$ is the second Stiefel-Whitney class. 
Since $H^*(Sp(k)/U(k))$ has no torsion, $H^*(Sp(k)/U(k);\Z/2\Z)\simeq H^*(Sp(k)/U(k);\Z)\otimes \Z/2\Z$.
We use the same symbol for an integral class and its mod $2$ reduction, and the equations below are meant to hold in
$H^*(Sp(k)/U(k);\Z/2\Z)$.
By Wu's formula, we have $Sq^2(c_i)=c_1c_i + (2i-1)(i-1) c_{i+1}$.
Since $c_i^2\in \Z/2\Z[q_1,q_2,\ldots,q_k]$, 
by Lemma \ref{lem:cohomology} we have 
\[
Sq^2 u_{2n-2} = Sq^2\left(\prod_{i=2}^k c_i\right) = (k-1)\prod_{i=1}^k c_i = u_{2n}.
\]
Set $c_1(\xi)=au_1$ and $c_{n-1}(\xi)=bu_{n-1}$ for some $a,b\in \Z$.
Since $k\equiv 2 \mod 4$, we have $n\equiv 1 \mod 2$.
Again by Wu's formula, we have
\[
bu_{n}= Sq^2(c_{n-1}(\xi)) = c_1(\xi) c_{n-1}(\xi)+c_n(\xi) =(ab+1)u_n.
\]
So $b(a+1)\equiv 1$, and hence, $a \equiv 0 \mod 2$.
\end{proof}

\begin{proof}[Proof of Theorem \ref{main-thm}]
Denote $Sp(k)/U(k)$ by $X$.
\begin{enumerate}
\item Consider the bundle 
\[
 \hat{\xi}=\left(\bigoplus_i L_{x_i} \right) \oplus \left(\bigoplus_{i<j} L_{x_i+x_j}\right)
\]
over $Sp(k)/T$.
Since $\hat{\xi}$ is invariant under the action of $W(U(k))$, there is a bundle $\xi$ over $Sp(k)/U(k)$ which 
pulls back to $\hat{\xi}$ via the projection $Sp(k)/T \to Sp(k)/U(k)$.
Then, $c_n(\xi)=\prod_i x_i \prod_{i<j} (x_i+x_j)=u_{2n}$ is a generator of the top degree cohomology
by Lemma \ref{lem:cohomology}. By Lemma \ref{lem:point}, the base point is represented by $\xi$.
\item
Assume that $\xi'\to X\times X$ represents the diagonal $\Delta(X)$.
By Proposition \ref{necessary} (2), the pullback of $\xi'$ along $\Delta: X\to X\times X$ is
isomorphic to the normal bundle $\nu(\Delta)$, which is isomorphic to $TX$.
On the other hand, the pullback of $\xi'$ along the inclusion to each factor
$i_1,i_2: X\to X\times X$ represents the class of the base point,
where $i_1(x)=(x,pt)$ and $i_2(x)=(pt,x)$.
Since $i_1^*\otimes i_2^*:H^2(X\times X)\simeq H^2(X)\otimes H^2(X)$, we see
\[
c_1(TX)=c_1(\Delta^*(\xi))=
\Delta^*(c_1(\xi))=
c_1(i_1^*(\xi))+c_1(i_2^*(\xi))\equiv 0 \mod 2
\] 
by Lemma \ref{lem:spin}.
However, $c_1(TX)=(k+1)u_1$ by Lemma \ref{lem:tangent} and
it contradicts that $k$ is even.
\end{enumerate}
\end{proof}

\red{
\begin{cor}
Let $G$ be of type $C_k$ and $P$ be of type $A$.
The base point in $G/P$ is representable.
\end{cor}
\begin{proof}
Note that any type $A$ parabolic subgroup $P$ is contained in
the maximal parabolic subgroup $P_k$ of type $A_{k-1}$. 
Apply Proposition \ref{bundle} to the flag bundle $P_k/P \hookrightarrow G/P \to G/P_k$,
where $G/P_k\simeq Sp(k)/U(k)$.
\end{proof}
}

\begin{rem}\label{rem:final}
A result of Totaro \cite{T2} shows
$\tau_{Spin(2k+1)/T}=\tau_{Spin(2k+2)/T}=2^{u(k)}$, where 
$u(k)$ is either $k - \lfloor \log_2( {k+1\choose 2} +1) \rfloor$
or that expression plus one. 
\red{Let $G$ be of type $B_k$ (resp. $D_{k+1}$) so that its compact form is $Spin(2k+1)$ (resp. $Spin(2k+2)$).
Since any parabolic subgroup of $G$ is 
a product of type $B$ (resp. of type $D$) and type $A$ subgroups,
the base point is not representable in $G/P$ for any $P$ when $u(k-1)<u(k)$
by the same argument as in the proof of Theorem \ref{exceptional}.
Note that $u(k-1)=u(k)$ rarely occurs when $k$ gets bigger. A list of $u(k)$ for small $k$ is given in \cite{T2}.
}

For example, let $Q_{l}=\{ x\in \C P^{l+1} \mid x_1^2+\cdots+x_{l+2}^2=0\}$ be the complex quadric.
In {\cite[Theorem 12]{PSP}},
it is shown that the diagonal in $Q_{l}$ is not representable for any odd $l$.
Since $Q_l$ is isomorphic to the real oriented Grassmannian (cf. {\cite[p.280]{KN}})
\[
\widetilde{Gr}_2(\R^{l+2}):=SO(l+2)/SO(2)\times SO(l),
\]
the base point in $Q_l$ is not representable for many $l$.
\red{
For example, $0=u(2)<u(3)=1$ shows that the base point in $Q_5$ is not representable 
as $\tau_{Spin(7)/T}\le \tau_{H/T} \cdot \tau_{Spin(7)/H}$, and hence, $2\le \tau_{Spin(7)/H}=\tau_{Q_5}$, where $H$ 
is the inverse image of $SO(2)\times SO(5)$ under the covering $Spin(7)\to SO(7)$.
Note the low rank equivalences $Q_1=\C P^1, Q_2=\C P^1 \times \C P^1, Q_3=Lag_\omega(\C^4)$, 
$Q_4=Gr_2(\C^4)$ and $Q_6=OG_4(\C^8)$.
So up to $l\le 6$, the base points are representable for $Q_1,Q_2,Q_3,Q_4$ but not for $Q_5$ and $Q_6$.
}

\red{
Theorem \ref{main-thm} tells that the converse to Proposition \ref{bundle} does not hold in general;
even when the diagonal of the total space of the type $A$ flag bundle 
$U(2)/T\to Sp(2)/T \to Sp(2)/U(2)$ is representable,
that of the base space is not representable.
This makes it difficult to complete the study of representability for partial flag manifolds of type $B,C$ and $D$.
The current status of the problem is summarised in the table below. 
The partial information obtained in this note on the entries with the symbol ``$?$''
suggests that a case-by-case analysis may be necessary to settle the remaining cases.
\begin{table}[hbt]
\begin{tabular}{|c|c|c|c|c|c|c|c|c}
\hline
type of $G$ & $A$ & $B$ & $C$ & $D$ & exceptional  \\ \hline 
point for $G/B$ & $\circ$ & $\times$ & $\circ$ & $\times$ & $\times$ \\
point for $G/P$ ($P$ of type $A$) & $\circ$ & $\times$ & $\circ$ & $\times$ & $\times$  \\
point for $G/P$ (otherwise) & $\circ$ & $?$ & $?$ & $?$ & $\times$  \\
diagonal for $G/B$ & $\circ$ & $\times$ & $\circ$ & $\times$  & $\times$\\
diagonal for $G/P$ ($P$ of type $A$) & $\circ$ & $\times$ & $?$ & $\times$ & $\times$ \\ 
diagonal for $G/P$ (otherwise) & $\circ$ & $?$ & $?$ & $?$ & $\times$ \\ \hline
\end{tabular}
\caption{Summary of representability}
\end{table}
}
\end{rem}

\end{document}